\documentclass[runningheads,a4paper]{llncs}

\usepackage{float}
\usepackage{verbatimbox}
\usepackage{mathtools}
\usepackage{amssymb}
\usepackage{graphicx}

\usepackage{url}
\newcommand{\keywords}[1]{\par\addvspace\baselineskip
\noindent\keywordname\enspace\ignorespaces#1}
\newcommand\articlecommand[1]{{\textsf{#1}}}
\newcommand{\inference}[3]{\ensuremath{\frac{~#2~}{~#3~}{\mbox{~\articlecommand{#1}}}}}
\newcommand{\squeezeup}{\vspace{-5mm}}

\begin{document}

\mainmatter  

\title{Arithmetic in Metamath, Case Study: Bertrand's Postulate}


\author{Mario Carneiro}
%

\institute{The Ohio State University, Columbus OH, USA}

\maketitle

\begin{abstract}
Unlike some other formal systems, the proof system Metamath has no built-in concept of ``decimal number'' in the sense that arbitrary digit strings are not recognized by the system without prior definition. We present a system of theorems and definitions and an algorithm to apply these as basic operations to perform arithmetic calculations with a number of steps proportional to an arbitrary-precision arithmetic calculation. We consider as case study the formal proof of Bertrand's postulate, which required the calculation of many small primes. Using a Mathematica implementation, we were able to complete the first formal proof in Metamath using numbers larger than 10. Applications to the mechanization of Metamath proofs are discussed, and a heuristic argument for the feasability of large proofs such as Tom Hales' proof of the Kepler conjecture is presented.
\keywords{Arithmetic $\cdot$ Metamath $\cdot$ Bertrand's postulate $\cdot$ decimal number $\cdot$ natural numbers $\cdot$ large proofs $\cdot$ ATP $\cdot$ Mathematica $\cdot$ formal proof}
\end{abstract}

\section{Introduction}\label{sec:intro}

The Metamath system, consisting of a formal proof language and computer verification software, was developed for the purpose of formalizing mathematics in a foundational theory which is as minimal as possible while still being able to express proofs ``efficiently'' (in a sense we will make more precise later) \cite{metamath}. Although Metamath supports arbitrary axiom systems, we are mainly concerned in this paper with the \texttt{set.mm} database, which formalizes much of the traditional mathematics curriculum into a ZFC-based axiomatization \cite{setmm}.

In this context, one can define a model of the real numbers and show that it satisfies the usual properties, and within this set we have the natural numbers $\bbbn$ and can give a name to the numbers $1,2,3,\ldots\in\bbbn$. One important point of comparison to other proof languages such as Mizar \cite{mizar} is that not every sequence of decimal digits is interpreted as a natural number.  In particular, the goal is to show rigorous derivations directly from ZFC axioms with complete transparency and not to depend on indirect proofs of correctness of, say, a computer arithmetic algorithm.  A sequence of digits is treated as any other identifier and only represents a natural number if it has been defined to do so.  The list of sequences so defined is quite short -- we have the definitions
$$2:=1+1,\qquad 3:=2+1,\qquad4:=3+1,\qquad5:=4+1,\qquad6:=5+1,$$
$$7:=6+1,\qquad 8:=7+1,\qquad9:=8+1,\qquad10:=9+1,$$
and this constitutes a complete listing of the defined integers in \texttt{set.mm} (not including $0$ and $1$, which are defined as part of the field operations).

The primary application of the \texttt{set.mm} has been for abstract math, and such small numbers were sufficient for prior theorems, but initial attempts at Bertrand's postulate showed that a new method was needed in order to systematize arithmetic on large numbers.

\subsection{Bertrand's Postulate}\label{sec:bpos}

\begin{theorem} [Chebyshev, Erd\H{o}s]\label{th:bpt}
For every $n\in\mathbb N$ there is a prime $p$ satisfying $n< p\le 2n$. 
\end{theorem}

This statement was conjectured by Joseph Bertrand in 1845, from which the problem gets its name, and it was proven in 1852 by Chebyshev. Our version of the theorem looks like this:
$$\inference{bpos}{}{\vdash n\in\bbbn\to\exists p\in\bbbp (n< p\land p\le 2n)}\footnote{The sans-serif labels mentioned in this paper refer to theorem statements in \texttt{set.mm}; they can be viewed at e.g. \url{http://us.metamath.org/mpegif/bpos.html} for \textsf{bpos}.}$$
Although the complete proof of this theorem is not the purpose of this paper, this was the motivating problem that led to the developments described here. For our formalization we targeted not Chebyshev's proof but rather a simpler proof due to Paul Erd\H{o}s. The proof is based on a detailed asymptotic analysis of the central binomial coefficients ${2n\choose n}$, and by keeping track of the bounds involved this shows Theorem~\ref{th:bpt} for $n>4000$. A usual exposition of the proof will observe that the other $4000$ cases can be shown by means of the sequence
\begin{equation}
2,3,5,7,13,23,43,83,139,163,317,631,1259,2503,4001,
\end{equation}
which can be seen to consist of primes $p_k$ such that $p_{k+1}<2p_k$.

It is this last statement which is the most problematic for a formal system which cannot handle large numbers, because verifying that large numbers are prime requires even larger calculations, and doing such calculations by hand on numbers of the form $1+1+\dots+1$ (as in our definition of the numbers up to $10$) is quite impractical, requiring $O(n^2)$ operations to multiply numbers of order $n$.

\subsection{Decimal Arithmetic}

To solve the problem of the space requirements of unary numbers, the standard approach is to use base-$10$ arithmetic, or more generally base-$b$ arithmetic for $b\ge2$, in which a nonnegative integer $n\in\bbbn_0$ is represented by an expression of the form
\begin{equation}
n=\sum_{i=0}^ka_ib^i=b\cdot(b\cdot (b\cdot a_k)+\dots+a_1)+a_0.
\end{equation}
Here $0\le a_i<b$, and arithmetic is performed relative to this representation, with ``addition with carry'' and long multiplication, using Horner's method in order to efficiently recurse through the structure of the expression. (We postpone the precise description of these (grade school) algorithms to Section~\ref{sec:mmth}.)

\subsubsection{All Calculation is Addition, Multiplication and Ordering of $\bbbn_0$.}

One important observation which is helpful to identify is that

\spnewtheorem{observation}{Observation}{\itshape}{\rmfamily}
\begin{observation}\label{ob:num}
All calculations of a numerical nature can be reduced to the three operations $x+y,x\cdot y,x<y$ applied to nonnegative integers.
\end{observation}
For instance:

\begin{equation}\label{eq:sqrt2}
\sqrt 2>1.414\quad\mbox{because}\quad1414\cdot1414<2\cdot 1000\cdot1000,
\end{equation}
\begin{equation}
11\mbox{ is prime}\quad\mbox{because}\quad 11=2\cdot 5+1=3\cdot 3+2,
\end{equation}
\begin{equation}\label{eq:pwbc}
\frac{4^4}4<{2\cdot 4\choose 4}\quad\mbox{because}\quad 4\cdot4\cdot4\cdot2\cdot3\cdot4<5\cdot6\cdot7\cdot8.
\end{equation}

Note that as this is a formal proof the goal is not in \emph{calculating} the result itself, which we already know or can verify externally, but rather in \emph{efficiently verifying the numerical claim}, which comes with its own challenges. For primality testing, this process is known as a primality certificate, and for larger primes we use so-called ``Pratt certificates'' for prime verification \cite{pratt}. In general, there may be many non-numerical steps involved to set up a problem into an appropriate form, such as the squaring and multiplication by $1000$ in equation~\ref{eq:sqrt2}, or we may want to add such steps as an additional reduction on the equation so that we avoid an unnecessarily complicated calculation, such as canceling the common factor $6\cdot 4$ in equation~\ref{eq:pwbc}. However, these ``framing'' steps are comparatively few so that it is worthwhile to coerce these calculations into the framework described here for even moderately large numbers, say $n>30$.

It is not really possible to prove Observation~\ref{ob:num} as it is not a strictly mathematical statement without additional explanation of what exactly is meant by ``calculation'' or ``numerical nature''. Its primary purpose is in justifying the scope of the arithmetic system to be built, and it is justified by the observation that computers can compute the various constants and special functions using only a small fixed instruction set and an arithmetic logic unit (ALU) that supports only these basic operations on bit strings. This observation can also be viewed as a combination of the Church-Turing thesis and the observation that arithmetic on $\bbbq$ can be expressed in the language of Peano Arithmetic.

Section~\ref{sec:mmth} describes the metatheory of ``decimal numbers'' as they appear in \texttt{set.mm}, and the theorems which form the basic operations upon which the algorithm is built. Section~\ref{sec:mma} describes the Mathematica implementation of a limited-domain automated theorem prover (ATP) for arithmetic using the decimal theorems as a base, and Section~\ref{sec:results} surveys the outcome of the project.

\section{Setting Up ``Decimal'' Arithmetic Theorems in \texttt{set.mm}}\label{sec:mmth}

\subsubsection{The State of the Database Before this Project.}

\begin{definition}\label{df:evaln}
When describing terms, we will use the notation $\langle x\rangle$ to refer to the defined number term with value $x$, assuming $0\le x\le10$, e.g. $3\cdot 3$ is a literal expression ``$3\cdot 3$'' while $\langle 3\cdot 3\rangle$ is the term $9$.
\end{definition}
We take for granted the following facts, already derived in the database:

\begin{itemize}
\item We have general theorems for the field operations, so that $a+b=b+a$, $a\cdot b=b\cdot a$, $a+0=a$, and $a\cdot 1=a$. We also have theorems of the form $x=x$;\ \ $x=y\vdash y=x$;\ \ $x=y,y=z\vdash x=z$ available.

\item As we have already mentioned in Section~\ref{sec:intro}, the numbers $0$-$10$ have been defined, yielding definitional theorems of the form $\langle n+1\rangle=n+1$ for each $n\in\{1,\dots,9\}$, and for each such number $n$ we have the theorems $n\in\bbbn_0$ and $n\in\bbbn$ (except $n=0$ which only has $0\in\bbbn_0$).

\item We have addition and multiplication facts for these numbers. That is, if $1\le n\le m\le 10$ and $m+n\le 10$ we have the theorem $m+n=\langle m+n\rangle$, and similarly if $1<n\le m\le 10$ and $mn\le 10$ then $m\cdot n=\langle mn\rangle$ is a theorem. By combining these facts with the general theorems on field operations we can construct the complete subsection of  ``addition table'' and ``multiplication table'' expressible using numbers less than $10$. (Recall that $6+5=11$ is not a theorem because even though $6+5$ is well-defined, ``$11$'' does not name a number and so the statement itself makes no sense unless $11$ is first given a definition.)

\item We have general inequality theorems like transitivity, and theorems of the form $m<n$ for each $0\le m<n\le10$.
\end{itemize}

This listing motivates the following definition:

\begin{definition}
A {\em basic fact} is a theorem of the form $x\in\bbbn$, $x\in\bbbn_0$, $x=y+z$, $x=y\cdot z$, or $x<y$ where $x,y,z$ are each number terms selected from $0,\dots,10$, which asserts a true statement about integers $x,y,z$.
\end{definition}

\begin{theorem}\label{th:basic}
There is an integer $N$ such that every basic fact is provable in less than $N$ steps.
\end{theorem}
\begin{proof}
Immediate since there are only finitely many basic facts.
\qed\end{proof}

It is not relevant for the analysis how large $N$ is, but by explicitly enumerating such facts one can show that $N\le10$, or $N\le3$ if closure steps $x\in\bbbn_0,\bbbr,\bbbc$ are ignored.

We begin by defining the concept of ``numeral''.

\begin{definition}\label{df:dec}
A {\em numeral} is a term of the language of {\normalfont\texttt{set.mm}} defined recursively by the following rules:
\begin{itemize}
\item The terms $0,1,2,3$ are numerals.

\item If $n$ is a nonzero numeral and $a\in\{0,1,2,3\}$, then the term $(4\cdot n)+a$ is a numeral.
\end{itemize}
\end{definition}

\subsubsection{Why Quaternary?} It is easily seen that this definition enumerates not base-$10$ integer representations but rather base-$4$ representation. It is of course necessary to commit to a base to work in for practical purposes, and the most logical decision for a system such as \texttt{set.mm} which prioritizes abstract reasoning over numerical calculation is base $10$. The reason this choice was rejected is because the ``multiplication table'' for base $10$ would require many more basic facts like $7\cdot8=10\cdot 5+6$, while the multiplication table of base $4$ requires only numbers as large as $9$, which fits inside our available collection of basic facts. Furthermore, within this constraint a large base allows for shorter representations, which directly translates to shorter proofs of closure properties and other algorithms (in addition to increasing readability). The specific choice of $4=2^2$ also works well with algorithms like exponentiation by squaring, used in the calculation of $x^k\bmod n$ in primality tests (see Section~\ref{sec:results}).

\begin{remark}\label{rem:meta}
It is important to recognize that ``$n$ is a numeral'' is not a statement of the object language, but rather of the metalanguage describing the actual structure of terms. Furthermore, these are terms for ``concrete'' integers, not variables over them, and there are even valid terms for nonnegative integers that are not equal to any numeral, such as ${\rm if}({\rm CH},1,0)$ (where ${\rm CH}$ is the Continuum Hypothesis or any other independent statement), which is provably a member of $\bbbn_0$ even though neither $\vdash{\rm if}({\rm CH},1,0)=0$ nor $\vdash{\rm if}({\rm CH},1,0)=1$ are theorems. However, we will show that these pathologies do not occur in the evaluation of multiplication, addition, and ordering on other numerals.
\end{remark}

There are three additional methods that are employed to shorten decimal representations, by adding clauses to Definition~\ref{df:dec}. We will call these ``extended numerals''.
\begin{itemize}
\item We can include the numbers $4,5,6,7,8,9,10$ as numerals. Although it is disruptive to some of the algorithms to allow these in the lower digits (e.g. considering $4\cdot 3+7$ to be a numeral), it is easy to convert such ``non-standard'' digits in the most significant place, via the (object language) theorems
$$\inference{dec4}{}{\vdash (4\cdot 1)+0=4},\quad\dots\quad\inference{dec10}{}{\vdash (4\cdot 2)+2=10}.$$
\item We can drop $0$ when it occurs in a numeral; this amounts to adding the rule ``if $n$ is a numeral then $4\cdot n$ is a numeral'' to the definition of a numeral. The conversion from this form to the usual form is provided by the theorem
$$\inference{dec0u}{\vdash n\in\bbbn_0}{\vdash 4n+0=4n}.$$
\item We can define a function $(x:y)=4x+y$ and add the rule ``if $x$ is a numeral and $y\in\{0,1,2,3\}$ then $(x:y)$ is a numeral''. Since multiplication and addition have explicit grouping in \texttt{set.mm}, this halves the number of parentheses and improves readability, from $((4\cdot ((4\cdot 3)+2))+1)$ to $((3:2):1)$. The conversion for this form is provided by the theorem
$$\inference{decfv}{\vdash x\in\bbbn_0\quad \vdash y\in\bbbn_0}{\vdash 4x+y=(x:y)}.$$
\end{itemize}

In remark~\ref{rem:meta} we observed that not all nonnegative integers are numerals, but we can show that the converse is true, so that in the above theorems, we can use the antecedent $n\in\bbbn_0$ instead of ``$n$ is a numeral'' which is not possible since this is not a statement of the object language. This weaker notion is sufficient to assure that $n$ has all the general properties we can expect of numerals, like $n+0=n$ or $n\ge0$, which is what we need for most of the theorems on numerals.

\begin{theorem}
If $n$ is a numeral, then $\vdash n\in\bbbn_0$.
\end{theorem}
\begin{proof}
By induction. If $n\in\{0,1,2,3\}$, then $n\in\bbbn_0$ is a basic fact. Otherwise, $n=4m+a$ for some numeral $m$ and $a\in\{0,1,2,3\}$, and by induction we can prove $m\in\bbbn_0$ and $a\in\bbbn_0$ is a basic fact. Then the result follows from application of the theorem
$$\inference{decclc}{\vdash m\in\bbbn_0\quad \vdash a\in\bbbn_0}{\vdash 4m+a\in\bbbn_0}.$$
\qed\end{proof}

\begin{theorem}\label{th:nncl}
If $n$ is a nonzero numeral, then $\vdash n\in\bbbn$.
\end{theorem}
\begin{proof}
By induction; the theorems involved are
$$\inference{decnncl2}{\vdash m\in\bbbn}{\vdash 4m+0\in\bbbn}\qquad \inference{decnnclc}{\vdash m\in\bbbn_0\quad \vdash a\in\bbbn}{\vdash 4m+a\in\bbbn}.$$
If $n\in\{1,2,3\}$, then $n\in\bbbn$ is a basic fact. Otherwise, $n=4m+a$ for some nonzero numeral $m$ and $a\in\{0,1,2,3\}$, and by induction $m\in\bbbn$. If $a=0$, then by \textsf{decnncl2} $n\in\bbbn$; otherwise $a\in\bbbn$ is a basic fact and $n\in\bbbn$ follows from \textsf{decnnclc}.
\qed\end{proof}

\begin{theorem}\label{th:eq}
The numeral representation is unique, in the sense that if $m=n$ as integers, then $m$ and $n$ are identical terms.
\end{theorem}
\begin{proof}
Follows from the uniqueness of base-$4$ representation of integers.
\qed\end{proof}

\begin{remark}\label{rem:eeq} Note that each of the extended representations of numerals invalidate this theorem, e.g. $(4\cdot2)+0=8=4\cdot2=(2:0)$ would all be distinct representations of the same integer value whose standard form is $(4\cdot2)+0$. Nevertheless, there is still a weaker form of uniqueness which is preserved. If $m,n$ are extended numerals and $m=n$ as integers, then $\vdash m=n$ is provable; this follows from the respective ``conversion'' theorems for each extension together with the equality theorems for addition and multiplication ($a=b,c=d\vdash a+b=c+d,\ a\cdot b=c\cdot d$).
\end{remark}

The converse of this, $\vdash m=n\implies m=n$, follows from soundness of ZFC from our interpretation of terms in the object logic as integers with the usual operations in the metalogic. This is why we will often not distinguish between equality in the metalogic and the object logic, because they coincide with each other and with identity as terms from Theorem~\ref{th:eq}.

\begin{theorem}\label{th:dec0}
If $n$ is a numeral then $\vdash 4a+b=n$ for some numeral $a$ (not necessarily nonzero) and some $b\in\{0,1,2,3\}$.
\end{theorem}
\begin{proof}
If $n\in\{0,1,2,3\}$, then we can use the theorem
$$\inference{dec0h}{\vdash a\in\bbbn_0}{\vdash 4\cdot 0+a=a},$$
since $n\in\bbbn_0$ is a basic fact. Otherwise, $n=4a+b$ for some nonzero decimal $a$ and $b\in\{0,1,2,3\}$, and the goal statement is just $\vdash 4a+b=4a+b$.
\qed\end{proof}

\begin{theorem}\label{th:lt}
If $m<n$ are numerals, then $\vdash m<n$.
\end{theorem}
\begin{proof}
By induction on $m$; the relevant theorems are
$$\inference{declti}{\vdash a\in\bbbn\quad \vdash b,c\in\bbbn_0\quad \vdash c<4}{\vdash c<4a+b},$$
$$\inference{declt}{\vdash a,b\in\bbbn_0\quad \vdash c\in\bbbn\quad \vdash b<c}{\vdash 4a+b<4a+c},$$
$$\inference{decltc}{\vdash a,b,c,d\in\bbbn_0\quad \vdash b<4\quad \vdash a<c}{\vdash 4a+b<4c+d}.$$
If $m,n\in\{0,1,2,3\}$, then $m<n$ is a basic fact. Otherwise if $m\in\{0,1,2,3\}$ and $n=4a+b$, then since $a$ is nonzero $a\in\bbbn$ by Theorem~\ref{th:nncl}, and $m<4$ is a basic fact, so \textsf{declti} applies to give $\vdash m<n$. If $m=4a+b$, then $m\ge 4$ so $n=4c+d$ (because $m<n$ implies $n\notin\{0,1,2,3\}$) for nonzero numerals $a,c$ and $b,d\in\{0,1,2,3\}$. If $a=c$, then by Theorem~\ref{th:eq}, $a$ and $c$ are identical, so (working in the metalogic) $4a+b<4a+d\to b<d$, and since $b,d\in\{0,1,2,3\}$ this is a basic fact; thus \textsf{declt} applies and $\vdash m<n$. Again working in the metalogic, if $a>c$ then $4a+b\ge4a\ge4c+4>4c+d$ in contradiction to the assumption $m<n$, so in the other case $a<c$ and by the induction hypothesis $\vdash a<c$; and $b<4$ is a basic fact, hence \textsf{decltc} applies.
\qed\end{proof}

As an example of Theorem~\ref{th:lt}, we know that $\vdash 3<4\cdot 3+1$ and $\vdash 4\cdot 2+0<4\cdot 2+1$ are theorems of \texttt{set.mm} because $3<13$ and $8<9$, respectively (and the numeral representations of these numbers are $[3]=3$, $[13]=4\cdot3+1$ and so on).

Next we show how to do successors, addition, and multiplication.
\begin{definition}\label{df:eval}
As a variant of Definition~\ref{df:evaln}, we use the notation $[x]$ to denote the unique numeral corresponding to the expression $x$. Thus for example $[6+5]=(4\cdot 2)+3$.
\end{definition}
\begin{remark}\label{rem:eval} Note that for $0\le x\le 10$, the statement $[x]=\langle x\rangle$ is either an identity or one of \textsf{dec4}, \textsf{dec5}, \dots, \textsf{dec10}, and so is provable in one step.
\end{remark}

\begin{theorem}\label{th:suc}
If $n$ is a numeral and $\vdash n=n'$, then $\vdash [n+1]=n'+1$.
\end{theorem}
\begin{proof}
By induction on $n$; the relevant theorems are
$$\inference{decsuc}{\vdash a,b\in\bbbn_0\quad \vdash c=(b+1)\quad \vdash 4a+b=n}{\vdash 4a+c=n+1},$$
$$\inference{decsucc2}{\vdash a\in\bbbn_0\quad \vdash b=(a+1)\quad \vdash 4a+3=n}{\vdash 4b+0=n+1}.$$
If $n\in\{0,1,2\}$, then $[n+1]\in\{1,2,3\}$ and $[n+1]=n+1$ is a basic fact, and $\vdash n+1=n'+1$. Otherwise, by Theorem~\ref{th:dec0} $n=4a+b$ for some (not necessarily nonzero) numeral $a$ and $b\in\{0,1,2,3\}$. If $b\in\{0,1,2\}$, then $[n+1]=4a+[b+1]$, and so the assumptions $[b+1]=b+1$, $4a+b=n'$ for \textsf{decsuc} are satisfied. Otherwise $b=3$ and $[n+1]=4[a+1]+0$, and by the induction hypothesis $\vdash[a+1]=a+1$ (using $n,n'\mapsto a$). Then $[a+1]=a+1$ and $4a+3=n'$ satisfy the assumptions of \textsf{decsucc2}. \qed\end{proof}

\begin{remark} The extra assumption $\vdash n=n'$ is not necessary (i.e. we could just have proven $\vdash [n+1]=n+1$) but makes it a little easier to work with extended numerals, because that way $n$ can be the standard numeral while $n'$ is the extended numeral, and the assumption is satisfied by remark~\ref{rem:eeq}.
\end{remark}

To give an example, if we were able to prove (using other theorems than discussed here) that $\vdash 4\cdot 1+1=6-1$, then Theorem~\ref{th:suc} says, taking $n=[5]=4\cdot 1+1$ and $n'=6-1$, that $\vdash 4\cdot 1+2=(6-1)+1$ is also provable.

\begin{theorem}\label{th:add}
If $m,n$ are numerals such that $\vdash m=m'$ and $\vdash n=n'$, then $\vdash [m+n]=m'+n'$.
\end{theorem}
\begin{proof}
By induction on $m+n$; the relevant theorems are
$$\inference{decadd}{\begin{array}{c}
  \vdash a,b,c,d\in\bbbn_0\quad \vdash 4a+b=m\quad \vdash 4c+d=n\\
  \vdash e=a+c\quad \vdash f=b+d\end{array}}{\vdash 4e+f=m+n},$$
$$\inference{decaddc}{\begin{array}{c}
  \vdash a,b,c,d,f\in\bbbn_0\quad \vdash 4a+b=m\quad \vdash 4c+d=n\\
  \vdash e=(a+c)+1\quad \vdash 4+f=b+d\end{array}}{\vdash 4e+f=m+n}.$$
If $m+n\in\{0,1,2,3\}$, then $[m+n]=m+n$ is a basic fact so $\vdash[m+n]=m'+n'$ follows from properties of equality. By Theorem~\ref{th:dec0} we can promote $m$ and $n$ to the form $m=4a+b$, $n=4c+d$ where $b,d<4$ and $a,c$ are numerals. Now if $b\,+\,d<4$, then $\vdash b\,+\,d<4$ follows from the basic facts $[b+d]=b+d$ and $[b+d]<4$, and so filling the assumptions with $4a+b=m'$, $4c+d=n'$, and $[b+d]=b+d$ and using the induction hypothesis to prove $[a+c]=a+c$, we can apply \textsf{decadd}. Otherwise, $b+d\ge4$, so $[b+d-4]\in\{0,1,2,3\}$ and we can use $4a+b=m'$, $4c+d=n'$, prove $4+[b+d-4]=b+d$ using the basic facts $\langle b+d\rangle=4+[b+d-4]$ and $\langle b+d\rangle=b+d$, and prove $[a+c+1]=[a+c]+1=(a+c)+1$ using first Theorem~\ref{th:suc} and then the induction hypothesis; this completes the assumptions of \textsf{decaddc}.
\qed\end{proof}

The next theorem works by double induction, so it is easier to split it into two parts.

\begin{theorem}\label{th:mul1}
If $m,n$ are numerals, $p\in\{0,1,2,3\}$, and $\vdash m=m',n=n'$, then $[mp+n]=m'p+n'$.
\end{theorem}
\begin{proof}
By induction on $m$, using the theorem
$$\inference{decmac}{\begin{array}{c}
  \vdash a,b,c,d,p,f,g\in\bbbn_0\quad \vdash 4a+b=m\quad \vdash 4c+d=n\\
  \vdash e=ap+(c+g)\quad \vdash 4g+f=bp+d\end{array}}{\vdash 4e+f=mp+n}.$$
If $m\in\{0,1,2,3\}$, then $[mp]=\langle mp\rangle$ is provable by remark~\ref{rem:eval} and $\langle mp\rangle=mp$ is a basic fact, so $\vdash[mp+n]=[mp]+n=m'p+n'$ by Theorem~\ref{th:add}. Otherwise let $m=4a+b$, and write $n=4c+d$ and $[bp+d]=4g+f$ for some $c,d,f,g$ by Theorem~\ref{th:dec0}. Then the assumptions to \textsf{decmac} are satisfied by $4a+b=m'$, $4c+d=n'$, $[ap+c+g]=ap+[c+g]=ap+(c+g)$ by the induction hypothesis and Theorem~\ref{th:add}, and $4g+f=[bp+d]=[bp]+d=bp+d$ by Theorem~\ref{th:add}.
\qed\end{proof}

\begin{theorem}\label{th:ma}
If $m,n,p$ are numerals, and $\vdash m=m',n=n',p=p'$, then $[mp+n]=m'p'+n'$.
\end{theorem}
\begin{proof}
By induction on $p$, using the theorem
$$\inference{decma2c}{\begin{array}{c}
  \vdash a,b,c,d,m,f,g\in\bbbn_0\quad \vdash 4a+b=p\quad \vdash 4c+d=n\\
  \vdash e=ma+(c+g)\quad \vdash 4g+f=mb+d\end{array}}{\vdash 4e+f=mp+n}.$$
If $p\in\{0,1,2,3\}$, then $[mp+n]=m'p+n'=m'p'+n'$ by Theorem~\ref{th:mul1}.  Otherwise let $m=4a+b$, and write $n=4c+d$ and $[mb+d]=4g+f$ for some $c,d,f,g$  by Theorem~\ref{th:dec0}. Then the assumptions to \textsf{decma2c} are satisfied by $4a+b=p'$, $4c+d=n'$, $[ma+c+g]=ma+[c+g]=ma+(c+g)$ by the induction hypothesis and Theorem~\ref{th:add}, and $4g+f=[mb+d]=mb+d$ by Theorem~\ref{th:mul1}.
\qed\end{proof}

\begin{corollary}
If $m,n$ are numerals, and $\vdash m=m',n=n'$, then $[mn]=m'n'$.
\end{corollary}
\begin{proof}
Theorem~\ref{th:ma} gives $[mn]=m'n'+0$, and $m'n'+0=m'n'$ because $m'n'=mn\in\bbbn_0$. (Slightly more efficient than this approach is an application of the theorems \textsf{decmul1c}, \textsf{decmul2c} which are essentially the same as \textsf{decmac}, \textsf{decma2c} without the addition component.)
\qed\end{proof}

\section{A Mathematica Implementation of the Decimal Arithmetic Algorithm}\label{sec:mma}

The purpose of the preceding section was not merely to prove that arithmetic operations are possible, which could be done just as easily using finite sums or unary representation. Rather, by proving the results contructively it is in effect a description of an algorithm for performing arithmetic calculations, and as such it is not difficult to implement on a computer. Due to its advanced pattern-matching capabilities, Mathematica was selected as the language of choice for the implementation.

In order to avoid Mathematica's automatic reduction of arithmetic expressions, we represent the Metamath formulas of our limited domain via the following correspondence:
\begin{itemize}
\item An integer $n$ is represented as itself
\item The term $a+b$ becomes $\mathtt{pl}[a,b]$
\item The term $a\cdot b$ becomes $\mathtt{tm}[a,b]$
\item $\vdash a=b$ becomes $\mathtt{eq}[a,b]$
\item $\vdash a<b$ becomes $\mathtt{lt}[a,b]$
\item $\vdash a\in\bbbn$ becomes $\mathtt{elN}[a]$
\item $\vdash a\in\bbbn_0$ becomes $\mathtt{elN0}[a]$
\item $\vdash a\in\bbbc$ becomes $\mathtt{elC}[a]$
\end{itemize}
These symbols have no evaluation semantics and so simply serve to store the shape of the target expression.

The output proof is stored as a tree of list expressions by the following correspondence:
\begin{quote} If the expression $e$ is obtained by applying the theorem $t$ to the list of expressions $e_1,\dots,e_k$ with proofs $p_1,\dots,p_k$, then the proof of $e$ is stored as the expression $p=\mathtt{\{\{}p_1,\dots,p_k\mathtt{\}},e,t\mathtt{\}}$.
\end{quote}

For example, the expression $4\cdot (4\cdot 1+3)+2=5\cdot 6$ is represented as \[\texttt{eq[pl[tm[4,pl[tm[4,1],3]],2],tm[5,6]]}\] and the proof of $2\cdot(4\cdot 1+1)\in\bbbn_0$ (by the sequence of theorems: $2\in\bbbn_0$ by \textsf{2nn0}, $1\in\bbbn_0$ by \textsf{1nn0}, $4\cdot 1+1\in\bbbn_0$ by \textsf{decclc}, $2\cdot(4\cdot 1+1)\in\bbbn_0$ by \textsf{nn0mulcli}) is:
\begin{verbbox}
{{{{}, elN0[2], "2nn0"},
  {{{{}, elN0[1], "1nn0"},
    {{}, elN0[1], "1nn0"}},
   elN0[pl[tm[4,1],1]], "decclc"}},
 elN0[tm[2,pl[tm[4,1],1]]], "nn0mulcli"}
\end{verbbox}
\begin{figure}[H]
  \centering
  \squeezeup
  \theverbbox
  \squeezeup
\end{figure}
(There are more efficient storage mechanisms, but this one is relatively easy to take apart and reorganize. Furthermore, since it only needs to run once in order to produce the proof, we are much more concerned with the length of the output proof than the speed of the proof generation itself.) For intermediate steps, we will also have use for the ``proof stubs'' \texttt{\{Null,}$e$\texttt{,"?"\}} (representing a proof with goal expression $e$ that has not been completed) and \texttt{\{\$Failed,}$e$\texttt{,"?"\}} (for a step that is impossible to prove or lies outside the domain of the prover).

Given a term expression, we can evaluate it easily using a pattern-matching function:
\begin{verbatim}
eval[pl[a_, b_]] := eval[a] + eval[b]
eval[tm[a_, b_]] := eval[a] eval[b]
eval[n_Integer] := n
\end{verbatim}
Then the domain of our theorem prover will be expressions of the form \texttt{eq[}$m$\texttt{,}$n$\texttt{]} where $m$ and $n$ are terms built from the integers $0$-$10$ and \texttt{pl}, \texttt{tm} which satisfy \texttt{eval[}$m$\texttt{] == eval[}$n$\texttt{]}. (As a side effect we will also support \texttt{elN0[}$n$\texttt{]}, and \texttt{elN[}$n$\texttt{]} when \texttt{eval[}$n$\texttt{] != 0}.) We will also need the reverse conversion:
\begin{verbatim}
bb[n_] := If[n < 4, n, pl[tm[4, bb[Quotient[n, 4]]], Mod[n, 4]]]
\end{verbatim}
This is the equivalent of the $[x]$ function from Definition~\ref{df:eval}.

Our algorithm works in reverse from a given goal expression, breaking it down into smaller pieces until we reach the basic facts. The easiest type of proof is closure in $\bbbn_0$:

\begin{verbatim}
prove[elN0[n_Integer]] := 
 With[{s = ToString[n]}, 
  If[n <= 4, {{}, elN0[n], s <> "nn0"}, {{prove[elN[n]]}, elN0[n], 
    "nnnn0i"}]]
prove[x : elN0[pl[tm[4, a_], b_]]] := {{prove@elN0[a], prove@elN0[b]},
   x, "decclc"}
prove[x : elN0[tm[a_, b_]]] := {{prove@elN0[a], prove@elN0[b]}, x, 
  "nn0mulcli"}
prove[x : elN0[pl[a_, b_]]] := {{prove@elN0[a], prove@elN0[b]}, x, 
  "nn0addcli"}
prove[elN[n_Integer]] /; n > 0 := {{}, elN[n], ToString[n] <> "nn"}
prove[elN[x : pl[y : tm[4, a_], b_Integer]]] := 
 If[b === 0, {{{prove@elN0[a], eq[x, y], "dec0u"}, prove@elN[y]}, 
   elN[x], "eqeltri"}, {{prove@elN0[a], prove@elN[b]}, elN[x], 
   "decnnclc"}]
prove[elN[x : tm[4, a_]]] := {{prove@elN[a]}, elN[x], "decnncl"}
\end{verbatim}
This simply breaks up an expression according to its head and applies \textsf{decclc} to numerals, \textsf{nn0addcli} and \textsf{nn0mulcli} to integer addition and multiplication, and theorems \textsf{0nn0}, \textsf{1nn0}, \dots, \textsf{4nn0}, \textsf{5nn}, \dots, \textsf{10nn} to the integers $0$-$10$ (where we switch to $\bbbn$ closure for numbers larger than $4$ because we do not have $\bbbn_0$ closure theorems prepared for these). Similarly, we can do closure for $\bbbn$:

\begin{verbatim}
prove[elN[n_Integer]] /; n > 0 := {{}, elN[n], ToString[n] <> "nn"}
prove[elN[x : pl[y : tm[4, a_], b_Integer]]] := 
 If[b === 0, {{prove@elN[a]}, elN[x], 
   "decnncl2"}, {{prove@elN0[a], prove@elN[b]}, elN[x], "decnnclc"}]
prove[elN[x : tm[4, a_]]] := {{prove@elN[a]}, elN[x], "decnncl"}
\end{verbatim}
This is just an implementation of Theorem~\ref{th:nncl}.

In order to do arbitrary equalities, we first ensure that both arguments are numerals, by chaining $x=[x]<[y]=y$ for inequalities and $x=[x]=y$ for equalities (where $[x]$ and $[y]$ are identical since we are assuming that the equality we are proving is in fact correct). We also allow the case ``$x<4$'' where $x$ is a numeral even though $4$ is not a numeral, because it comes up often and we already have theorems for this case.
\begin{verbatim}
prove[lt[x_, y_]] /; eval[x] < eval[y] := 
 If[x === bb@eval[x], 
  If[y === bb@eval[y] || y === 4, 
   provelt[x, y], {{prove@lt[x, bb@eval[y]], proveeq2[y]}, lt[x, y], 
    "breqtri"}], {{proveeq2[x], prove@lt[bb@eval[x], y]}, lt[x, y], 
   "eqbrtrri"}]
prove[eq[x_, y_]] /; eval[x] == eval[y] := 
 If[x === bb@eval[x], 
  proveeq2[y], {{proveeq2[x], proveeq2[y]}, eq[x, y], "eqtr3i"}]
proveeq2[x_] := proveeq[bb@eval[x], x]
\end{verbatim}

It remains to define \texttt{provelt} and \texttt{proveeq}. We start with \texttt{provelt}, implementing Theorem~\ref{th:lt}.
\begin{verbatim}
provelt[x_Integer, y_Integer] := {{}, lt[x, y], 
  ToString[x] <> "lt" <> ToString[y]}
provelt[x_Integer, 
  y : pl[tm[4, a_], b_]] := {{prove@elN[a], prove@elN0[b], 
   prove@elN0[x], provelt[x, 4]}, lt[x, y], "declti"}
provelt[x : pl[tm[4, a_], c_], y : pl[tm[4, b_], d_]] := 
 If[a === b, {{prove@elN0[a], prove@elN0[c], prove@elN[d], 
    provelt[c, d]}, lt[x, y], "declt"},
  {{prove@elN0[a], prove@elN0[b], prove@elN0[c], 
    prove@elN0[d], provelt[c, 4], provelt[a, b]}, lt[x, y], "decltc"}]
\end{verbatim}

The contract of \texttt{proveeq}$[x,y]$ is such that it returns a proof of $\vdash x=y$ given an expression $y$, assuming $x=[y]$ (since we can calculate $x$ from $y$, the left argument is not necessary, as in the variant \texttt{proveeq2}, but it simplifies pattern matching), and it is defined much the same as previous functions; we elide it here due to space constraints.

\subsubsection{Basic facts.} One fine point which may need addressing, since it was largely glossed over in Theorem~\ref{th:basic}, is the algorithm for basic facts. We define a function \texttt{basiceq[x\_]} which is valid when \texttt{x} is either \texttt{pl}$[m,n]$ or \texttt{tm}$[m,n]$ and \texttt{eval[x]}$\le10$; in this case it corresponds to a ``basic fact'' of either addition or multiplication, and the return value is a proof of \texttt{eq[eval[x], x]}. There are many special cases, but every such statement follows from \textsf{addid1i}, \textsf{addid2i} (addition with $0$), \textsf{mulid1i}, \textsf{mulid2i} (multiplication by $1$), \textsf{mul01i}, \textsf{mul02i} (multiplication by $0$), \textsf{df-2}, \dots, \textsf{df-10} (addition with $1$), or a named theorem like \textsf{3p2e6} for $3\cdot2=6$ -- these exist for each valid triple containing numbers larger than $1$ -- possibly followed with \textsf{addcomi}, \textsf{mulcomi} (commutation) and/or \textsf{eqcomi}, \textsf{eqtri} (symmetry/transitivity of equality) to tie the components together.

\section{Results}\label{sec:results}

In this section we will describe the purpose to which we applied the arithmetic algorithm.

\subsection{Prime Numbers}
There are several ways to prove that a number is prime, and the relative efficiency can depend a lot on the size of the numbers involved. For small numbers, especially if it is necessary to find all primes below a cutoff, the most efficient method is simple trial division. The biggest improvement in efficiency here is gained by looking only at primes less than $\sqrt n$ in a proof that $n$ is prime. Starting from the two primes $2,3$ which are proven ``from first principles'', we can use this to show that a number less than $25$ which is not divisible by $2$ or $3$ is prime, and we reduce these primality/compositeness deductions to integer statements via
$$\inference{ndvdsi}{\vdash q\in\bbbn_0\quad \vdash a,r\in\bbbn\quad \vdash b=aq+r\quad \vdash r<a}{\vdash a\nmid b}$$
$$\inference{nprmi}{\vdash a,b\in\bbbn\quad \vdash 1<a\quad \vdash 1<b\quad \vdash n=ab}{\vdash n\notin\bbbp}.$$
As a consequence of our choice of base $4$, we also have easy proofs of compositeness or non-divisibility by $2$:
$$\inference{dec2dvds1}{\vdash a\in\bbbn_0}{\vdash 2\nmid4a+1}\quad \inference{dec2dvds3}{\vdash a\in\bbbn_0}{\vdash 2\nmid4a+3}$$
$$\inference{dec2nprm}{\vdash a\in\bbbn}{\vdash 4a+2\notin\bbbp}$$

This yields proofs of $5,7,11,13,17,19,23\in\bbbp$. We repeat the process to show that a number less than $29^2=841$ and which is not divisible by $2,3,5,7,11,13,17,\allowbreak 19,23$ is prime (the upper bound is $29^2$ because $29$ is the first prime larger than $23$), and we use this theorem to prove primality of $37,43,83,139,163,317,631$. There are three more primes needed for the prime sequence in Bertrand's postulate, namely $1259,2503,4001$, and we would need many more primes to repeat the process with a still-larger upper bound, so we switch methods.

\subsubsection{Pocklington's Theorem.}
\begin{theorem} If $N>1$ is an integer such that $N-1=AB$ with $A>B$, and for every prime factor $p$ of $A$ there is an $a_p$ such that $a_p^{N-1}\equiv 1\pmod N$ and $\gcd(a_p^{(N-1)/p}-1,N)=1$, then $N$ is prime.
\end{theorem}

This theorem has been proven in \texttt{set.mm} in decimal-friendly form (in the case when $A=p^e$ is a prime power) as
$$\inference{pockthi}{\begin{array}{c}
  \vdash p\in\bbbp\quad \vdash g,B,e,a\in\bbbn\quad \vdash m=gp\quad \vdash N=m+1\quad \vdash m=Bp^e\\
  \vdash B<p^e\quad \vdash a^m\equiv 1\pmod N\quad \vdash\gcd(a^g-1,N)=1\end{array}}{\vdash N\in\bbbp}.$$

Ignoring the $p^e$ term which is easily evaluated by writing it out as a product of $p$'s since $p^e<N$ is relatively small, there are two new kinds of integer statements involved here, $a\equiv b\pmod N$ and $\gcd(c,d)=e$. We can further narrow our concern to statements of the form $a^m\equiv b\pmod N$ and $\gcd(c,d)=1$ (where $c\ge d$), and we can reduce the second via Euclid's algorithm in the form
$$\inference{gcdi}{\vdash k,r,n\in\bbbn_0\quad \vdash m=kn+r\quad \vdash \gcd(n,r)=g}{\vdash \gcd(m,n)=g}.$$
For the power mod calculations, we used addition-chain exponentiation on manually selected chains which had good calculational properties (small intermediate calculations), since the hardest step in the calculation
$$\inference{modxai}{\vdash e=b+c\quad \vdash dn+m=kl\quad \vdash a^b\equiv k,\ a^c\equiv l \pmod n}{\vdash a^e\equiv m\pmod n}\footnote{closure assumptions elided}$$
is evaluating the expression $dn+m=kl$, whose proof length is driven by the size of $k,l$, so that exponents with smaller reduced forms yield shorter proofs. In the worst case, for $N=4001$ we are calculating products on the order $kl\le N^2=16\,008\,001$, which are roughly $12$-digit numbers in base $4$.

\subsection{Bertrand's Postulate Gets the Last Laugh}
As mentioned in Section~\ref{sec:bpos}, the framework described in this paper was developed in preparation for performing the large calculations needed to prove that numbers like $4001$ are prime. After this work was completed, we discovered that there was a new proof by Shigenori Tochiori \cite{tochiori} (unfortunately untranslated to my knowledge) which, by strengthening the estimates in Erd\H{o}s' proof, manages to prove the asymptotic part for $n\ge 64$ (instead of $n\ge4000$), so that the explicit enumeration of primes above this became unnecessary for the completion of the proof. Nevertheless, the proof of $4001\in\bbbp$ remains as a good example of a complicated arithmetic proof, and we expect that this arithmetic system will make it much easier to handle such problems in the future, and \textsf{bpos} is now completed in any case.

\subsection{Large Proofs}
One recent formal proof which has gathered some attention is Thomas Hales' proof of the Kepler Conjecture, also known as the Flyspeck project \cite{flyspeck}, and it highlights one foundational issue regarding Metamath's prospects in the QED vision of the future \cite{qed} -- \emph{which} and \emph{how much} resources are stressed by projects like this with a large computational component? The design of Metamath is such that it can verify a proof in nearly linear time, assuming that it can store the entire database of theorem statements in memory, because at each step it need only verify that the substitutions to the theorem statement are in fact done correctly (and the substitutions themselves are stored as part of the proof, even though they can be automatically derived with more sophisticated and slower algorithms). Thus the primary bottleneck in verification time is the length of the proof itself, and we can analyze this quite easily for our chosen algorithm.

Since we are essentially employing grade-school addition and multiplication algorithms, it is easy to see that they are $O(n)$ and $O(n^2)$ respectively, and with more advanced multiplication algorithms we could lower that to $O(n^{1.5})$ or lower. Indeed, there does not seem to be any essential difference between the number of steps in a Metamath proof and the number of cycles that a computer might go through to perform the equivalent algorithm, even though a Metamath proof doesn't ``run'' per se as it is a proof and not a program. (In fact a Metamath proof has at least one big advantage over a computer in that it can ``guess the right answer'' in the manner of a non-deterministic Turing machine.) To take this example to its conceptual extreme, we could even simulate an ALU with addition and multiplication of integers representing data values of a computer, and then the progress of the proof would directly correspond to the steps in a computer program. Of course this would introduce a ridiculously large constant, but it would suggest that any program, including a verifier for another formal system such as the HOL Light system in which Flyspeck runs, can be emulated with a proof whose length is comparable to the running time of the verifier without a change in the overall asymptotics.

\subsubsection*{Acknowledgments.} The author wishes to thank Norman Megill for discussions leading to the developments of Section~\ref{sec:mmth}, N. Megill and Stefan O'Rear for reviewing early drafts of this work, and the many online Japanese language learning resources that assisted in translating and eventually formalizing \cite{tochiori}.

\end{document}